\documentclass[a4paper,12pt]{article}
\usepackage{amsmath,amsfonts,enumerate,amssymb,amsthm,color}
\usepackage{pgf,tikz}
\title{Regular Cayley maps on dihedral groups with the smallest kernel}
\author{Istv\'an Kov\'acs,$^{a}$   \; Young Soo Kwon$^{\, b}$ \\  [+0.75ex]
$^a$ {\small IAM, University of Primorska, Muzejski trg 2, 6000 Koper, Slovenia} \\ [-0.5ex]
$^b$ {\small Department of Mathematics, Yeungnam University, Kyongsan 712-749, Republic of Korea}
}
\date{}

\newtheorem{thm}{Theorem}[section]
\newtheorem{lem}[thm]{Lemma}
\newtheorem{cor}[thm]{Corollary}

\theoremstyle{definition}

\theoremstyle{remark}
\newtheorem{rem}[thm]{Remark}

\def\D{\mathbb{D}}
\def\M{\mathcal{M}}

\def\proof{\noindent{\bf Proof.}\ }
\def\QED{\hfill $\square$ \medskip }
\def\Z{\mathbb{Z}}
\DeclareMathOperator{\aut}{Aut}
\DeclareMathOperator{\cay}{Cay}
\DeclareMathOperator{\cm}{CM}
\DeclareMathOperator{\core}{Core}
\DeclareMathOperator{\cosg}{Cos}
\DeclareMathOperator{\kr}{Ker}

\DeclareMathOperator{\sym}{Sym}

\newcommand{\comment}[1]{}

\begin{document}

\maketitle
\let\thefootnote\relax\footnote{
The work was partially supported by the Slovenian-Korean bilateral
project, grant no.\ BI-KOR/13-14-002. The first author was also
supported by the ARRS grant no.\ P1-0285. The second author was
supported by the 2014 Yeungnam University Research Grant.  \\
[+0.5ex] {\it  E-mail addresses:} istvan.kovacs@upr.si (Istv\'an
Kov\'acs), ysookwon@ynu.ac.kr (Young Soo Kwon). }

\begin{abstract}
Let $\M=\cm(D_n,X,p)$ be a regular Cayley map on the dihedral group
$D_n$ of order $2n, n \ge 2,$ and let $\pi$ be the power function
associated with $\M$. In this paper it is shown that the kernel
$\kr(\pi)$ of the power function $\pi$ is a dihedral subgroup of
$D_n$ and  if $n \ne 3,$ then the kernel $\kr(\pi)$ is of order at
least $4$. Moreover,  all $\M$ are classified for which $\kr(\pi)$
is of order $4$. In particular, besides $4$ sporadic maps on $4,4,8$
and $12$ vertices respectively, two infinite families of
non-$t$-balanced  Cayley maps on $D_n$ are obtained.

\medskip\noindent{\it Keywords:} regular map, regular Cayley map,
skew-morphism, dihedral group.

\medskip\noindent{\it MSC 2010:}  05C10, 05C30.
\end{abstract}

\section{Introduction}

In this paper all groups are finite, and all graphs are finite,
simple and connected. For a graph $\Gamma$, we let $V(\Gamma)$,
$E(\Gamma),$ $D(\Gamma),$ and $\aut(\Gamma)$ denote the vertex set,
the edge set, the dart (or arc) set, and the full group of
automorphisms of $\Gamma$, respectively. By a \emph{map} with an
underlying graph $\Gamma$ we mean a triple $\M=(\Gamma;R,T),$ where
$R$ is a permutation of the dart set $D(\Gamma)$ whose orbits
coincide with the  sets of darts initiating in the same vertex, and
$T$ is an involution of $D(\Gamma)$ whose orbits coincide with sets
of darts with the same underlying edge. The permutations $R$ and $T$
are called the \emph{rotation} and the \emph{dart-reversing
involution} of $\M$, respectively. Given two maps
$\M_i=(\Gamma_i;R_i,T_i),$ $i=1,2,$ an \emph{isomorphism} $\Phi :
\M_1 \to \M_2$ is a bijection $\Phi : D(\Gamma_1) \to  D(\Gamma_2)$
such that $\Phi R_1 =  R_2 \Phi$ and $\Phi T_1 = T_2 \Phi$.  In
particular, if $\M_1 = \M_2 = \M,$ then $\Phi$ is called an
\emph{automorphism}, and the group of all automorphisms of $\M$ will
be denoted by $\aut(\M)$. It is easily seen that $\aut(\M)$ acts
semi-regularly on the dart set $D(\Gamma),$ and in the case when
this action is also transitive the map $\M$ is called
\emph{regular}. In what follows the map $(\Gamma;R,T)$ will be
written as the pair $(\Gamma;R)$ because $T$ is uniquely defined by
$\Gamma$. For more information on regular maps we refer the reader
to the survey paper \cite{N01}.  \medskip

Let $G$ be a group and let $X$ be a generating set of $G$ such that
$X = X^{-1}$ and $1_G \notin X,$ where $1_G$ denotes the identity of
$G$. The \emph{Cayley graph} $\cay(G,X)$ is the graph with vertex
set $G$ and with edges in the form $\{g,g x\}, g \in G, x \in X$.
The \emph{left multiplication} $L_g$ induced by $g \in G$ is the
permutation of $G$ defined by $L_g(h) =gh$ for any $h \in G$. We set
$L(G) = \{L_g : g \in G\}$. It is clear that $L(G) \le
\aut(\cay(G,S))$. Let $p$ be a cyclic permutation of $X$. The
\emph{Cayley map} $\cm(G,X,p)$ is the map $(\Gamma;R)$ with
underlying graph $\Gamma = \cay(G,X)$ and rotation $R$ defined by $R
: (g,g x) \mapsto (g,g p(x)), g \in G, x \in X$. It can be easily
checked that for every $g \in G,$ $L_g R = R L_g$ and $L_g T =  T
L_g,$ so $L(G)$ is also a subgroup of $\aut(\M)$ acting regularly on
the vertex set. Two Cayley maps $\M_i = \cm(G_i,X_i,p_i), i=1,2,$
are called \emph{equivalent}, denoted by $\M_1 \equiv \M_2,$ if
there exists a group isomorphism $\phi : G_1 \to G_2$  mapping $X_1$
to $X_2$ such that $\phi p_1 = p_2 \phi$. Equivalent Cayley maps are
isomorphic as maps. The converse, however, does not hold in general,
i.e., there exist isomorphic Cayley maps which are not equivalent.

The class of cyclic groups is the only class of finite groups on which all regular Cayley maps have been classified due to the work of
Conder and Tucker \cite{CT14}. Regarding other groups, only partial classifications are
known (see, e.g. \cite{KMM13,KKF06,KO08,O09,WF05,Z14,Z}).
For more information on regular Cayley maps, the reader is referred to \cite{B72,JS02,RSJTW05,SS92,SS94}. \medskip

In this paper we focus on regular Cayley maps on dihedral groups.
The dihedral group of order $2n$ for $n \ge 2$ will be denoted by
$D_n$.  A complete classification of regular Cayley maps on dihedral
groups have been given in \cite{WF05} for balanced maps; in
\cite{KKF06} for $t$-balanced maps with $t > 1;$ in \cite{KMM13} for
non-balanced maps with $n$ being an odd number; and in \cite{Z} for
maps of skew-type $3$. Recall that a Cayley map $\M = \cm(G,X,p)$ is
\emph{$t$-balanced} if $p(x)^{-1} = p^t(x^{-1})$ for every $x \in
X$. In particular, if $t=1$ then $\M$ is called \emph{balanced}, and
if $t=-1,$ then $\M$ is called \emph{anti-balanced}. Let $\pi$ be
the power function associated with a regular Cayley map $\M =
\cm(G,X,p)$  (for the definition of $\pi,$ see 2.1). Let $\kr(\pi) =
\{g \in G : \pi(g)=1\}$ be the kernel of the power function $\pi$.
Following \cite{Z14}, we also say that $\M$  is of \emph{skew-type
$k$ when $|G : \kr(\pi)|=k$.} If $\M$ is $t$-balanced, then it was
proved to be of skew-type at most $2$ in \cite{CJT07,SS92}. More
precisely, $t=1$  holds if and only if $\kr(\pi) =G$ (see
\cite{SS92}); and if $t > 1$ and $G=D_n,$ then $\kr(\pi)$ is a
dihedral subgroup of $D_n$ of index $2$. In this context, the papers
\cite{KKF06,WF05,Z} deal with regular Cayley maps $\M =
\cm(D_n,X,p)$ having a large kernel.

In this paper we consider the other extreme case, i.e., the
associated kernel is as small as possible. We are going to prove
that, if $\M = \cm(D_n,X,p)$ is a regular Cayley map with associated
power function $\pi,$ then either $\M$ is the embedding of the
octahedron into the sphere and $|\kr(\pi)| = 2,$ or $|\kr(\pi)| \ge
4$ (see Theorem~\ref{min}). Moreover, we are also  going to
determine those maps $\M$ for which $|\kr(\pi)| = 4$. In this paper
we set $D_n = \langle a,b \mid a^n=b^2=baba=1 \rangle$ and
$C_n=\langle a \rangle$. Note that if $n
>  2,$ then $C_n$ is the unique cyclic subgroup of $D_n$ of order
$n$.

\begin{thm}\label{main}
Let $\M$ be a regular Cayley map on $D_n$ such that $|\kr(\pi)| = 4$
for the associated power function $\pi$. Then exactly one of the
 following holds:
\begin{enumerate}[(1)]
\item $n=2,$ and $\M \equiv \cm(D_2,\{a,b\},(a,b))$ or
$\cm(D_2,\{a,b,ab\},(a,b,ab))$.
\item $n=4,$ and $\M \equiv \cm(D_4,\{a,a^{-1},b\},(a,b,a^{-1}))$.
\item $n=6,$ and $\M \equiv \cm(D_6,\{a,a^{-1},ab,a^{-1}b\},(a,a^{-1},ab,a^{-1}b)$.
\item $n=2m,$ $n \ge 6,$
$\M \equiv \cm(D_n,a \langle a^2 \rangle \, \cup \, b \langle a^2 \rangle, p)$ with
$$
p = (b,a,a^2 b,a^3,a^4 b,\ldots,a^{n-2}b,a^{n-1}).
$$
\item $n=2m,$ $8 \mid n,$ $\M \equiv \cm(D_n,a \langle a^2 \rangle \, \cup \, b \langle a^2 \rangle, p)$ with
$$
p = (b,a,a^{m+2}b,a^3,a^4 b,\ldots,a^{m-2}b,a^{n-1}).
$$
\end{enumerate}
\end{thm}

\section{Preliminaries}

In this section we collect all concepts and results  needed in this paper.
\paragraph{2.1  Skew-morphisms of finite groups.}
For a finite group $G,$ let $\psi : G \to G$ be a permutation of
the underlying set $G$ of order $r$ (in the full symmetric group $\sym(G)$) and let $\pi : G \to \{1,\ldots,r\}$
be any function. The permutation $\psi$ is a \emph{skew-morphism} of $G$
with \emph{power function} $\pi$ if $\psi(1_G)=1_G,$ and $\psi(g h) = \psi(g) \psi^{\pi(g)}(h)$ for all $g,h \in G$.
Skew-morphisms were defined by Jajcay and \v Sir\'a\v n in \cite{JS02}, where the following theorem was shown:

\begin{thm}\label{JS}
A Cayley map $\M = \cm(G,X,p)$ is regular if and only if there
exists a skew-morphism $\psi$ of $G$ such that $\psi(x) = p(x)$ for
all $x\in X$.
\end{thm}

The skew-morphism $\psi$ and its power function $\pi$ in the above theorem are
uniquely determined by the regular Cayley map $\M$.
In what follows these will be referred to as the \emph{skew-morphism (power function) associated with} $\M$.
More precisely, for a given regular Cayley map $\cm(G,X,p),$ the associated skew-morphism $\psi$ is of order $|X|,$
and the distribution of the values of $\pi$ on $X$ is given by
the following formula
(see \cite{JS02}):
\begin{equation}\label{eq1}
\pi(x) \equiv \chi(\psi(x)) - \chi(x) + 1(\text{mod }|X|) \text{ for any } x \in X,
\end{equation}
where $\chi(x)$ is the smallest non-negative integer such that $p^{\chi(x)}(x) = x^{-1}$ (notice that $x^{-1} \in X$ as
$X = X^{-1}$ holds).  \medskip

The \emph{kernel} of the power function $\pi$ is defined by
$\kr(\pi) = \{g \in G : \pi(g) = 1\}$. The following lemma shows
some basic properties (see \cite{JS02}):

\begin{lem}\label{JS-lem}
Let $\psi$ be a skew morphism of $G$ and let $\pi$ be
the corresponding power function of $\psi$.
\begin{enumerate}[(1)]
\item $\kr(\pi)$ is a subgroup of $G$.
\item $\pi(g) = \pi(h)$ if and only if $g$ and $h$ belong to the same right coset of $\kr(\pi)$.
\item $\pi(gh) \equiv \displaystyle{\sum_{i=0}^{\pi(g)-1} \pi(\psi^i(h))}
(\text{mod }r),$ where $r$ is the order of $\psi$.
\end{enumerate}
\end{lem}

For a Cayley map $\M = \cm(G,X,p),$ we will denote by
$\aut(\M)_{1_G}$ the stabilizer of the vertex $1_G$ in the group
$\aut(\M)$ in its action on the vertices. Notice that if $\M$ is
regular, then $\aut(\M)_{1_G}$ is generated by the skew-morphism
$\psi$ associated with $\M,$ and so $\aut(\M)$ admits the
factorization $\aut(\M)  = L(G) \, \aut(\M)_{1_G} = L(G) \, \langle
\psi \rangle$.

\paragraph{2.2 $\mathbf{G}$-arc-regular Cayley graphs on dihedral groups.}

Let $\Gamma$ be a graph and let $G \le \aut(\Gamma)$. Then $\Gamma$
is called  \emph{$G$-arc-regular} if $G$ is regular on the dart set
$D(\Gamma)$. Clearly, if $\M = (\Gamma;R)$ is a regular map, then
the underlying graph $\Gamma$ is $\aut(\M)$-arc-regular.  \medskip

Let $n=2m,$ $m$ is an odd number. For the rest of the paper we set $\D_n$ for the group
$\D_n = (D_m \times D_m) \rtimes \langle \sigma \rangle,$
where $\sigma$ is an involution of $\D_n$ which acts on $D_m \times D_m$ by switching the coordinates, i.e.,
$\D_n = \langle D_m \times D_m, \sigma \rangle,$ and
\begin{equation}\label{eq2}
\sigma (d_1,d_2) \sigma = (d_2,d_1) \text{ for all } (d_1,d_2) \in D_m \times D_m.
\end{equation}

The \emph{core} of a subgroup $A \le B$ in a group $B,$ denoted by $\core_B(A),$ is the largest normal subgroup of
$B$ contained  in $A$. The subgroup $A$ is \emph{core-free} in $B$ if $\core_B(A)$ is trivial. \medskip

The following result of Kov\'acs et al.\  \cite{KMM13} will be one of our main tools in
this paper (see \cite[Theorem~2.8]{KMM13}):

\begin{thm}\label{KMM}
Let $\Gamma = \cay(D_n,S)$ be a connected $G$-arc-regular graph such that $L(D_n) \le G$, and
every cyclic subgroup of $L(D_n)$ of order $n$ is core-free in $G$. Then one of the following holds:
\begin{enumerate}[(1)]
\item $n=1$, $\Gamma \cong K_2$, and $G \cong S_2$,
\item $n=2$, $\Gamma \cong K_4$, and $G \cong A_4$,
\item $n=3$, $\Gamma \cong K_{2,2,2}$, and $G \cong S_4$, 
\item $n=4$, $\Gamma\cong Q_3$, and $G\cong S_4$,
\item $n=2m$, $m$ is an odd number, $\Gamma \cong K_{n,n}$, and $G \cong \D_n$. Moreover,
the subgroup of $\D_n$ corresponding to $L(D_n)$ is contained in $D_m \times D_m$.
\end{enumerate}
\end{thm}

\paragraph{2.3 Quotient Cayley maps.}

Let $\M = \cm(G,X,p)$ be a regular Cayley map. Suppose, in addition,
that there exists a subgroup $N \le G$ such that $N$  is normal in
$G$ and the the set of $N$-cosets is a block system of $\aut(\M)$.
In what follows it will be simply said that $G/N$ is a block system
for $\aut(\M)$. Furthermore, we set $X/N = \{ N x : x \in X \}$.
Clearly, $X/N$ is a generating subset of the factor group $G/N$ and
$X/N = (X/N) ^{-1}$.
 Also, since $\cay(G,X)$ is $
\aut(\M)$-arc-regular, no element of $X$ belongs to $N$, and so
$1_{G/N} \notin X/N$. \medskip

There is an action of $\aut(\M)$ on the set of blocks, i.e., on
$G/N$. For $g \in \aut(\M),$ we let $\bar{g}$ denote the action of
$g$ on $G/N$, and for a subgroup $H \le \aut(\M)$ set $H^{G/N} =\{
\overline{g} : g \in H \}$. Notice that $(L_g)^{G/N} = L_{N g}$ for
every $g \in G$. Let us write $X = \{x_1,\ldots,x_{k}\}$ and $p =
(x_1,x_2,\ldots,x_{k})$. Then it follows that the cycle $p^{G/N} : =
(x_1 N,x_2 N,\ldots,x_{k} N)$ is well-defined (see \cite{KMM13});
and so is the Cayley map $\cm(G/N,X/N,p^{G/N})$. The latter map is
called  the \emph{quotient} of $\M$ with respect to the block system
$G/N,$ and it will be also denoted by $\M/N$. We note that  the
quotient map $\M/N$ coincides with the so called
\emph{Cayley-quotient} induced by the normal subgroup $N$ which was
defined by Zhang \cite{Z}, and in the same paper $\M$ is also
referred to as the \emph{Cayley-cover} of $\M/N$.
 We collect below some properties (see \cite[Corollary~3.5]{KMM13}):

\begin{lem}\label{quo}
Let $\M = \cm(G,X,p)$ be a regular Cayley map with associated
skew-morphism $\psi$ and power function $\pi,$ and let $N \le  G$ be
a normal subgroup in $G$ and $G/N$ is a block system for $\aut(\M)$.
Then the following hold:
\begin{enumerate}[(1)]
\item $\M/N = \cm(G/N,X/N,p^{G/N})$ is also regular.
\item $\aut(\M/N) = \aut(\M)^{G/N}$.
\item The skew-morphisms associated  with $\M/N$ is equal to $\psi^{G/N}$.
\item The order
$|\langle \psi \rangle| \le |N| \cdot \big|\langle \psi^{G/N}
\rangle\big|,$  and equality holds if and only if $X$ is a union of
$N$-cosets.
\item The power function $\pi^{G/N}$ associated with $\M/N$ satisfies
$$
\pi^{G/N}(Ng) \equiv \pi(g)\big(\text{mod } \big|\langle \psi^{G/N}
\rangle\big| \big) \text{ for every } g \in G.
$$
\end{enumerate}
\end{lem}

\section{Regular Cayley maps with a given group}

Let $G$ be a finite group, $H$ be a non-trivial subgroup of $G,$ and let $x,y$ be elements  in $G$ such that $y \ne 1_G$.
We say that the ordered quadruple $(G,H,x,y)$ is \emph{admissible} if  the following properties hold:
\begin{itemize}
\item  $G = H Y$ and $|H \cap Y|= 1,$ where $Y = \langle y \rangle;$
\item $Y$ is  core-free in $G;$
\item $G = \langle Y, x \rangle$ and $Y x Y = Y x^{-1} Y$.
\end{itemize}
Every admissible quadruple $(G,H,x,y)$ gives rise to a regular Cayley map on $H$ defined as follows. \medskip

First, recall that the \emph{coset graph} $\Gamma = \cosg(G,Y,YxY)$
has vertex set $G/Y,$ the set of left $Y$-cosets in $G,$ and its
edges are in the form $\{ g_1 Y,g_2 Y\}, g_1,g_2 \in G$ and
$g_1^{-1} g_2 \in Y x Y$. Note that the edges are well-defined
because of the condition $Y x Y = Y x^{-1} Y$. Also, the condition
$G = \langle Y, x \rangle$ makes $\Gamma$ to be connected. Since the
group $Y$ is core-free in $G,$ the action of $G$ on the set $G/Y$ is
a faithful permutation representation of $G$. Furthermore, the dart
set $D(\Gamma)$ is, in fact, equal to the orbit of the dart $(Y,x
Y)$ under $G$.  \medskip

Now, using that $G = H Y$ and $|H \cap Y| = 1,$ there is a bijection
from $G/Y$ onto $H$. Observe that this  bijection induces an
isomorphism from $\Gamma$ to the Cayley graph $\cay(H,X),$ where $X$
is the subset of $H$ defined by
$$
X = \{ h \in H : h Y \subseteq Y x  Y\}.
$$
Notice that $X$ is the unique subset of $H$ that satisfies $X Y = Y
x Y$.

Also, by the above bijection we obtain a faithful permutation representation of $G$ on $H$.
More precisely, an element $g \in G$ acts on $H$ by letting $g(h)$ to be the unique element of $H$ for which
\begin{equation}\label{eq3}
g(h) Y = gh Y, \text{ where } g \in G, h \in H.
\end{equation}
Furthermore, $X$ becomes the orbit of $x$ under $Y$ in the above action. Hence we can define the cyclic permutation $p$ of
$X$ as
$$
p = \big( x, y(x), y^2(x), \ldots, y^{|Y|-1}(x) \big).
$$
Now, the Cayley map $\cm(H,X,p)$ will be called the \emph{Cayley map induced by} $(G,H,x,y),$ and in what follows we will
write $\cm(G,H,x,y)$ for $\cm(H,X,p)$.

\begin{lem}\label{aq1}
Let $(G,H,x,y)$ be an admissible quadruple and let $\M = \cm(G,H,x,y)$ be the Cayley map on $H$ induced by
$(G,H,x,y)$. Then the following hold:
\begin{enumerate}[(1)]
\item $\M$ is regular, and $\aut(\M) \cong G$.
\item If $\alpha : G \to \widehat{G}$ is an isomorphism, then the quadruple
$(\alpha(G),\alpha(H),\alpha(x),\alpha(y))$ is also admissible; moreover, $\M$ and
$\M(\alpha(G),\alpha(H),\alpha(x),\alpha(y))$ are equivalent.
\end{enumerate}
\end{lem}

\proof Let us consider the action of $G$ on $H$ defined in
\eqref{eq3}. Let $\psi_y$ denote the
permutation of $H$ describing the action of $y$. Then  $p(x) =
\psi_y(x)$ for every $x \in X,$ and thus case (1) of the lemma
follows if we can prove that $\psi_y$ is a skew-morphism of $H$ (see
Theorem~\ref{JS} and the remark after Lemma~\ref{JS-lem}).

Notice that the permutation of $H$ describing the action of $h \in
H$ is equal to the left multiplication $L_h$. Thus the permutation
subgroup corresponding to the action of $G$ factorizes as $L(H)
\langle \psi_y \rangle,$ in particular, $L(H) \langle \psi_y \rangle
= \langle \psi_y \rangle L(H)$.

We compute next $\psi_y(1_H)$. By \eqref{eq3}, $\psi_y(1_H) Y =  y1_H Y = Y,$ hence $\psi_y(1_H) \in H \cap Y,$ and so
$\psi_y(1_H) = 1_H$.

Pick an arbitrary $h \in H$. Since $\langle \psi_y \rangle L(H) =
L(H) \langle \psi_y \rangle,$ $\psi_y L_h = L_{h'} \psi_y^i$ for a
unique $h' \in H$ and a unique $i \in \{1,\ldots,|Y|\}$. Notice that
$i$ depends entirely on $h,$ and thus we may define the function
$\pi : H \to \{1,\ldots,|Y|\}$ by letting $\pi(h) = i$. Also,
$(\psi_y L_h)(1_H)=\psi_y(h)$ and $(L_{h'} \psi_y^i)(1_H) = h'$.
These give that $h' = \psi_y(h)$. Thus if $h_1,h_2 \in H,$ then we
may write
$$
\psi_y(h_1 h_2) = (\psi_y L_{h_1})(h_2) = (L_{\psi_y(h_1)} \psi_y^{\pi(h_1)})(h_2) = \psi_y(h_1) \psi_y^{\pi(h_1)}(h_2),
$$
showing that $\psi_y$ is indeed a skew-morphism of $H$ with power function $\pi$.
So case (1) of the lemma is proved.  \medskip

We turn to case (2).  It is obvious that all defining axioms of an admissible quadruple
are preserved by $\alpha,$ and so
$(\alpha(G),\alpha(H),\alpha(x),\alpha(y))$ is also admissible.

By definition, $\cm(G,H,x,y) = \cm(H,X,p),$ where $X$ is the unique subset of $H$ satisfying
$X Y = Y x Y;$ and for $z \in X,$ $p(z)$ is the unique element in $H$ satisfying
$p(z) Y = y z Y$.  Also, $\cm(\alpha(G),\alpha(H),\alpha(x),\alpha(y)) =
\cm(\alpha(H),\widehat{X},\widehat{p}),$ where
$\widehat{X}$ is the unique subset of $\alpha(H)$ satisfying
$\widehat{X} \alpha(Y) = \alpha(Y) \alpha(x) \alpha(Y);$ and for $\widehat{z} \in \widehat{X},$
$\widehat{p}(\widehat{z})$ is the unique element in $\alpha(H)$ satisfying
$\widehat{p}(\widehat{z}) \alpha(Y) = \alpha(y) \widehat{z} \alpha(Y)$.

Obviously, $\alpha(X) \alpha(Y) = \alpha(Y) \alpha(x) \alpha(Y)$.
This implies that $\widehat{X} = \alpha(X),$ i.e., $\alpha$ maps $X$
onto $\widehat{X}$. For any $z \in X$, it follows from $p(z) Y = y z
Y$ that
$$
\alpha(p(z)) \alpha(Y) = \alpha(p(z)Y) = \alpha(yzY) =  \alpha(y)
\alpha(z) \alpha(Y).
$$
Since $\alpha(p(z)) \in \alpha(H),$ we have $\alpha(p(z)) =
\widehat{p}(\alpha(z))$. We conclude that $\alpha p = \widehat{p}
\alpha,$ and so $\cm(H,X,p)$ and
$\cm(\alpha(H),\widehat{X},\widehat{p})$ are equivalent. Case (2) of
the lemma is proved. \QED

\begin{rem}\label{rem1}
Let $\varphi$ be an arbitrary permutation of $H$ such that
$\varphi(1_H) = 1_H$. Now, it becomes apparent from the above proof
that a sufficient condition for $\varphi$ to be a skew-morphism of
$H$ is that $\varphi L(H) \subseteq L(H) \langle \varphi \rangle$.
\end{rem}

\begin{lem}\label{aq2}
Let $\M = \cm(H,X,p)$ be a regular Cayley map such that $\aut(\M)
\cong G$. Now there exists an admissible quadruple $(G,K,x,y)$ such
that $H \cong K$ and $\M$ and $\cm(G,K,x,y)$ are equivalent.
\end{lem}

\proof  Let $\psi$ be the skew-morphism associated with $\M$ and let
$\Psi = \langle \psi \rangle$. The group $\aut(\M)$ factorizes as
$\aut(\M) = L(H) \Psi$ such that $|L(H) \cap \Psi|=1,$ see the
remark after Lemma~\ref{JS-lem}. Also, by $\Psi$ being the
stabilizer of the vertex $1_H$ in $\aut(\M),$ the group $\Psi$  is
core-free in $\aut(\M)$.  Let us fix an element $x_1 \in X$. Using
that $X$ is a generating set of $H,$ it is not hard to show that
$\langle\Psi,L_{x_1} \rangle$ is transitive on $H$.  Using this and
that $\langle \Psi, L_{x_1} \rangle \ge \Psi = \aut(\M)_{1_H},$  we
obtain that $\langle\Psi,L_{x_1} \rangle = \aut(\M)$. Since $X =
X^{-1},$ $x_1^{-1}$ also belongs to $X$, and hence $\Psi L_{x_1}
\Psi = \Psi L_{x_1}^{-1} \Psi$. We may summarize all these as the
quadruple $(\aut(\M),L(H),L_{x_1},\psi)$ is admissible.
\medskip

Denote by $L$ the isomorphism from $H$ to $L(H)$ defined by $L: h \mapsto L_h$.
Let $h$ be an element in $H$ such that $L_h \Psi \subset \Psi L_{x_1} \Psi$. Then $L_h = \psi^i L_{x_1} \psi^j$ for some integers $i$ and $j,$
and so $h = L_h(1_H) = (\psi^i L_{x_1} \psi^j)(1_H) = \psi^j(x_1)$. As $X$ is the orbit of $x_1$ under $\Psi$ we
see that $h \in X$. On the other hand, for any $h \in X,$ $h = \psi^k(x_1)$ for some integer $k$. This  implies that the product
$L_h^{-1} \psi^k L_{x_1}$ fixes $1_H,$ and so $L_h^{-1} \psi^k L_{x_1}  \Psi = \Psi,$ namely $L_h \Psi \subset \Psi L_{x_1} \Psi$.
Thus
\begin{equation}\label{eq4}
\cm(\aut(\M),L(H),L_{x_1},\psi) = \cm(L(H),L(X),\widehat{p}),
\end{equation}
where for $z \in X,$ $\widehat{p}(L_z)$ is the unique element of
$L(H)$ satisfying $\widehat{p}(L_z) \Psi = \psi L_z \Psi$. As $(\psi
L_z)(1_H)= \psi(z)$ and $L_{p(z)}(1_H) = p(z)= \psi(z),$ we have
$L_{p(z)} \Psi = \psi L_z \Psi$. Since $L_{p(z)} \in L(H),$
$L_{p(z)} = \widehat{p}(L_z)$. We conclude that $L p=\widehat{p} L,$
and so $\cm(H,X,p)$  and $\cm((L(H),L(X),\widehat{p})$ are
equivalent.

Let $\alpha$ be an isomorphism $\alpha : \aut(\M) \to G$.
Then by \eqref{eq4} and Lemma~\ref{aq1}(2),
$\cm(L(H),L(X),\widehat{p})$ is also equivalent with $\cm(G,\alpha(L(H)),\alpha(L_{x_1}),\alpha(\psi))$.
The lemma follows by choosing $K = \alpha(L(H)),$ $x=\alpha(L_{x_1}),$ and $y=\alpha(\psi)$. \QED

We combine the above lemmas with Theorem~\ref{KMM} to have the following theorem.

\begin{thm}\label{c-free}
Let $\M$ be a regular Cayley map on $D_n$ such that the unique
cyclic subgroup of $L(D_n)$ of order $n$ is core-free in $\aut(\M)$
with $n \ge 5$. Then $n=2m,$ $m$ is an odd number, $\M \equiv
\cm(D_n,a \langle a^2 \rangle \, \cup \, b \langle a^2 \rangle, p),$
where
$$
p = (b,a,a^2 b,a^3,a^4 b,\ldots,a^{n-2}b,a^{n-1}).
$$
\end{thm}

\proof The Cayley graph $\cay(D_n,X)$ is $\aut(\M)$-arc-regular, and thus
Theorem~\ref{KMM} is applicable.
This gives that $n=2m,$ $m$ is an odd number,
$\aut(\M) \cong \D_n,$ and the subgroup of $\D_n$ corresponding to $L(D_n)$ is contained in $D_m \times D_m$.
The proof of  Lemma~\ref{aq2} yields that
\begin{equation}\label{eq5}
\M \equiv \cm(\D_n,H,x_1,y_1)
\end{equation}
for some  admissible quadruple $(\D_n,H,x_1,y_1)$ such that $H \cong D_n$ and $H \le D_m \times D_m$. \medskip

Let us fix $c$ as a generator of the cyclic subgroup of $D_m$ of order $m,$ and
an involution $r \in D_m$ such that $r \notin \langle c \rangle$.
Define the subgroup $D \le D_m \times D_m$ as
$$
D = \{ (d,r^{i}) : d \in D_m, i \in \{0,1\}  \}.
$$
Then $D \cong D_m \times \Z_2 \cong D_n$. It was proved in
\cite[Proposition~3.2]{KMM12} that, every two  subgroups of $D_m
\times D_m,$ which are isomorphic to $D_n,$ are conjugate in $\D_n$.
Thus $D = H^g$ for some $g \in \D_n$.
 Apply Lemma~\ref{aq1}(2) to $\cm(\D_n,H,x_1,y_1)$
with letting $\alpha$  be the inner automorphism $\alpha : z \mapsto
z^g, z \in \D_n$. This and \eqref{eq5} imply that
\begin{equation}\label{eq6}
\M \equiv \cm(\D_n,D,x_2,y_2), \;  x_2 = x_1^g  \text{ and } y_2= y_1^g.
\end{equation}

Let $Y_2 = \langle y_2 \rangle$. Recall that $Y_2$ is core-free in
$\D_n,$ $\D_n = D Y_2,$ and $|D \cap Y_2|=1$. Since $\D_n = D Y_2$
and $D \le (D_m \times D_m),$ $Y_2 \not\leq (D_m \times D_m)$. We
obtain that $|Y_2| = |\D_n : D|=n,$ and $|Y_2 \cap (D_m \times D_m)|
= |Y_2|/2 = m$. Since $m$ is odd, the unique  involution $y_2^m$ in
$Y_2$ is in the form  $y_2^m=(d,d')\sigma$ for some $(d,d') \in D_m
\times D_m$.  Then $((d,d')\sigma)^2=y_2^{2m}=1_G$, hence
$d'=d^{-1}$. This gives  $y_2^m= (d,1)\sigma(d,1)^{-1},$ i.e.,
$y_2^m$ is conjugate to $\sigma$ by $(d,1) \in D$. This and
\eqref{eq6} imply that
\begin{equation}\label{eq7}
\M \equiv \M(\D_n,D,x,y), \; x = x_2^{(d,1)} \text{ and } y = (y_2)^{(d,1)}.
\end{equation}

Notice that $y^m = \sigma$.  Also, $y=(d_1,d_2)\sigma$ for some
$(d_1,d_2) \in D_m \times D_m$. As $y$ commutes with $\sigma$ and
has order $n$, we find that $d_1=d_2 = c^{\ell}$  for some integer
$\ell$ with $\gcd(\ell,m)=1$. Thus we have
\begin{equation}\label{eq8}
Y = \langle y \rangle = \big\{  (c^i,c^i)\sigma^j : i \in \{0,1,\ldots,m-1\}, j \in \{0,1\} \big\}.
\end{equation}

We compute next the skew-morphism $\psi_y$ of $D$ obtained from the action of $y$ defined in \eqref{eq3}.
We set $a = (c,r)$ and $b = (r,1),$ and so
$D = \langle a,b \mid a^n=b^2=1, bab=a^{-1} \rangle$.
Let $i \in \{0,1,\ldots,m-1\}$. Using \eqref{eq2} and \eqref{eq8}, we may write
\begin{eqnarray*}
y a^{2i} Y  &=&  (c^\ell,c^\ell) \sigma (c^{2i},1) Y =  (c^{-2i},1) (c,c)^{2i+\ell} \sigma Y = (c^{n-2i},1) Y = a^{n-2i} Y, \\
y (a^{2i+1}b) Y &=&  (c^\ell,c^\ell) \sigma (c^{2i+1}r,r) Y =
(c^\ell r,c^{\ell+2i+1}r)\sigma Y = (c^{-2i-1}r,r) Y = a^{n-2i-1}b Y, \\
y a^{2i+1} Y     &=& (c^\ell,c^\ell) \sigma (c^{2i+1},r) Y =
(c^\ell r,c^{\ell+2i+1})\sigma Y = (c^{2\ell+2i+1}r,1) Y =
a^{2i+1+2\ell+m}b Y, \\
y (a^{2i}b) Y     &=& (c^\ell,c^\ell) \sigma (c^{2i}r,1) Y =
(c^\ell,c^{\ell+2i}r)\sigma Y = (c^{2\ell+2i},r) Y = a^{2i+2\ell+m} Y.
\end{eqnarray*}
Therefore, $\psi_y$ is given by
$$
\psi_y(a^j) = \begin{cases} a^{n-j} & \mbox{ if $j$ is even} \\  a^{j+2\ell+m}b & \mbox{ if $j$ is odd} \end{cases} \quad
\psi_y(a^jb) = \begin{cases} a^{j+2\ell+m} & \mbox{ if $j$ is even} \\
a^{n-j}b & \mbox{ if $j$ is odd}. \end{cases}
$$
Notice that the set $X = a \langle a^2 \rangle \, \cup \, b \langle
a^2 \rangle$ is the only orbit of $\psi_y$ which generates
$D_n$. This and \eqref{eq7}  imply that $\M \equiv \cm(D,X,p_1),$ where $p_1(z) = \psi_y(z)$ for every $z \in X$. 

Since $\gcd(\ell,m)=1$ and $m$ is odd, it follows  that, $\gcd(2\ell+m,m)=1$. This implies that  $\gcd(2\ell+m,n)=1$ because  $2\ell+m$ is odd.
Thus there is an integer $\ell'$ such that $(2\ell+m) \ell' \equiv 1(\text{mod }n)$.
Define the automorphism $\alpha$ of $D$ by letting $\alpha(a) = a^{\ell'}$ and
$\alpha(b) =b$. Now, it is easily seen that $\alpha$ induces an equivalence from $\cm(D,X,p_1)$ to
$\cm(D,X,p),$ where $p = (b,a,a^2 b,a^3,a^4 b,\ldots,a^{n-2}b,a^{n-1}).$
The theorem is  proved. \QED

\section{The  smallest kernel of regular Cayley maps on $\mathbf{D_n}$}

In this section, we show that every kernel of regular Cayley maps on
$D_n$ is a dihedral subgroup of $D_n$. Furthermore such kernel is of
order at least $4$ except only one regular Cayley maps on $D_3$. We
start with a simple observation.
\medskip

\begin{lem}\label{ker}
Let $\M=\cm(G,X,p)$  be a regular Cayley map with associated skew-morphism $\psi$ and power function $\pi$. Then $\kr(\pi) \cong L(G) \cap L(G)^\psi$.
\end{lem}

\proof Let $L$ be the isomorphism $L : G \to L(G)$ defined by $L : g
\mapsto L_g$. Since $\psi$ is a skew-morphism with power function
$\pi,$ $\psi L_g = L_{\psi(g)} \psi^{\pi(g)}$ holds in $\aut(\M)$
for every $g \in G$.  The  equivalences follow:
$$
g \in \kr(\pi) \iff  L_g^{\psi^{-1}} = \psi L_g \psi^{-1} \in L(G) \iff L_g \in L(G)^\psi.
$$
Therefore, $L$ maps the group $\kr(\pi)$ onto $L(G) \cap L(G)^\psi$. The lemma is proved. \QED

\begin{cor}\label{cor}
Let $\M=\cm(D_n,X,p)$  be a regular Cayley map with associated skew-morphism $\psi$ and power function $\pi$ and let $M=C_n \cap \kr(\pi)$.
If $|M| > 2,$ then $L(M)$ is normal in $\aut(\M)$.
\end{cor}

\begin{proof}
Let $L$ be the isomorphism $L : G \to L(G)$ defined by $L : g \mapsto L_g$.
We have proved above that $L$ maps $M$ to $L(D_n) \cap L(D_n)^\psi$.
Thus $L(M) \le L(D_n)^\psi,$ and hence $L(M)^{\psi^{-1}} \le L(D_n)$.
Since $|M| > 2,$ $L(M)$ is the unique cyclic subgroup of $L(D_n)$ of order $|M|,$
and we conclude that $L(M)^\psi=L(M)$. Therefore, $L(M)$ is normal
in $\langle L(D_n),\psi \rangle = \aut(\M)$. \QED
\end{proof}

\begin{thm}\label{min}
Let $n \ge 2,$ and $\M=\cm(D_n,X,p)$  be a regular Cayley map with
associated skew-morphism $\psi$ and power function $\pi$. Now
\begin{enumerate}[(1)]
\item $\kr(\pi)$ is a dihedral subgroup of $D_n$.
\item Either $\M$ is the
embedding of the octahedron into the sphere and $|\kr(\pi)| = 2,$ or
$|\kr(\pi)| \ge 4$.
\end{enumerate}
\end{thm}

\proof   Let $N$ be the subgroup of $C_n$ such that $L(N)$ is the
core of $L(C_n)$ in $\aut(\M)$. Notice that $N \le \kr(\pi)$. If $N$
is trivial, then the results hold by Theorem \ref{KMM}. So assume
that $N$ is non-trivial, namely $|N| \ge 2$. Since  $N \le
\kr(\pi)$, it suffices to show that $\kr(\pi)$ is a dihedral
subgroup of $D_n$, namely $\kr(\pi) \cap (D_n \setminus C_n) \neq
\emptyset$.

Let us consider the  largest subgroup $H \le C_n$ containing $N$
such that $D_n/H$ is a block system for $\aut(\M)$.  Now if $H=C_n,$
then $L(D_n)$ is normal in $\aut(\M),$ $\kr(\pi) = D_n,$ which prove
the results. Hence assume that $H < C_n$. We set $\Gamma =
\cay(D_n,X)$ and $\Gamma/H = \cay(D_n/H,X/H)$.

Recall that $\Gamma/H$  is the underlying graph  of the quotient map
$\M/H$ (see 2.3 and Lemma~\ref{quo}). Now $\M/H$ is a regular Cayley
map on the dihedral group $D_n/H,$ and $\Gamma/H$ is
$\aut(\M/H)$-arc-regular. Since $\aut(\M/H) = \aut(\M)^{D_n/H},$ it
follows by the maximality of $H$ that the core of $L(C_n/H)$ is
trivial in $\aut(\M/H)$. Theorem~\ref{KMM} is applicable, in
particular, $\Gamma/H$ and $\aut(\M/H)$ are described in one of
cases (2)-(5) of Theorem~\ref{KMM}.

If $\Gamma/H$ and $\aut(\M/H)$ correspond to  (2) or (3) in
Theorem~\ref{KMM}, then there exists a $x \in X \cap (D_n \setminus
C_n)$ such that $\psi(x) \in D_n \setminus C_n$. This means that $x
\in \kr(\pi)$, and hence $\kr(\pi)$ is a dihedral subgroup of $D_n$.

Suppose that $\Gamma/H$ and $\aut(\M/H)$ correspond to  (4) or (5)
in Theorem~\ref{KMM}. Then there exists $x \in X \cap C_n$ such that
$\psi(x), \psi^{-1}(x^{-1}) \in X \cap (D_n \setminus C_n)$. If
$p^i(x^{-1})=x$, then $\pi(x)=\pi(\psi^{-1}(x^{-1}))=i+1$. Since $x
\in C_n$ and $\psi^{-1}(x^{-1}) \in D_n \setminus C_n$, $\kr(\pi)$
is a dihedral subgroup of $D_n$ by Lemma \ref{JS-lem}(2). This
completes the proof of the theorem. \QED

We study next the Cayley maps defined in cases (4) and (5) of Theorem~\ref{main}, respectively.

\begin{lem}\label{map1}
Let $n=2m,$ $n \ge 6,$ $\M = \cm(D_n,a \langle a^2 \rangle \, \cup \, b \langle a^2 \rangle, p),$
where  $p = (b,a,a^2 b,a^3,$ $a^4 b,\ldots,a^{n-2}b,a^{n-1})$. Then the following hold:
\begin{enumerate}[(1)]
\item $\M$ is regular, and its associated skew-morphism $\psi$ is given by
$$
\psi(a^j) = \begin{cases} a^{-j} & \mbox{ if $j$ is even} \\  a^{j+1}b & \mbox{ if $j$ is odd} \end{cases} \quad
\psi(a^jb) = \begin{cases} a^{j+1} & \mbox{ if $j$ is even} \\  a^{-j}b & \mbox{ if $j$ is odd.} \end{cases}
$$
\item Let $N$ be the subgroup of $C_n$ such that $L(N)$ is the core of $L(C_n)$ in $\aut(\M)$. Then
$|N| \le 2,$ and $|N|=1$ if and only if $m$ is odd.
\item The associated power function $\pi$ has kernel $\kr(\pi) \cong \Z_2^2$.
\item Let $\Pi = \{ \pi(g) : g \in D_n\}$. Now $\Pi = \{2i+1 : i \in \{0,1,\ldots,m-1\} \}$. If $4 \mid n,$ then
\begin{equation}\label{eq9}
\pi(x) \equiv -1(\text{mod }4) \iff x \in a \langle a^2 \rangle \, \cup \, b \langle a^2 \rangle.
\end{equation}
\end{enumerate}
\end{lem}

\proof We consider  all points step by step. \medskip

\noindent (1): By Remark~\ref{rem1}, it is sufficient to show that $\psi L(D_n) \subseteq  L(D_n) \langle \psi \rangle$.

It can be checked by definition that $\psi^2$  is
 skew-morphism of $D_n$ whose power function takes $1$ on all $x  \in \langle a^2,b \rangle,$ and $-1$ on the
remaining elements. Hence $\psi^2$ acts on $\langle a^2,ab \rangle$
as the identical permutation,  and on the rest as the left
multiplication $L_{a^2}$.  Direct computations yield
$$
L_{a^2b} \psi L_a \psi^{-1} = \psi^2 \text{ and }  L_{a^{-1}} \psi L_b \psi^{-1} = \psi^{-2}.
$$
Now we may write $\psi L_{a^j} \psi^{-1} = (\psi L_a \psi^{-1})^j =
(L_{a^2b} \psi^2)^j \in L(D_n) \langle \psi^2 \rangle,$ and thus
$\psi L_{a^j} \in L(D_n) \langle \psi^2 \rangle \psi \subseteq
L(D_n) \langle \psi \rangle$. Similarly, $\psi L_{a^j b} \psi^{-1} =
(\psi L_{a^j} \psi^{-1}) (\psi L_b \psi^{-1}) =
 (L_{a^2b} \psi^2)^j L_a \psi^{-2} \in L(D_n) \langle \psi^2 \rangle,$
and so $\psi L_{a^j b} \in L(D_n) \langle \psi^2 \rangle \psi \subseteq L(D_n) \langle \psi \rangle$. \medskip

\noindent (2): Suppose for the moment that $m$ is even. It is then
easily seen that $L_{a^m} \psi = \psi L_{a^m},$ and thus $L_{a^m}
\in L(N)$. By this observation it suffices to prove that, if  $|N|
\ge 2,$ then $m$ is even and $|N|=2$.

Let $L(N)$ be generated by $L_{a^k}$ where $k$ is a divisor of
$n=2m$. Now $(L_{a^k})^\psi \in L(N)$ and $(L_{a^k})^\psi(1_H) =
(\psi^{-1} L_{a^k} \psi)(1_H) = \psi^{-1}(a^k)$. This  implies in
turn that, $k$ is even, $(L_{a^k})^\psi(1_H) = a^{-k},$ and so
$(L_{a^k})^\psi = L_{a^{-k}}$. Then $a^{-k}b = L_{a^{-k}}(b) =
(L_{a^k})^\psi(b)=a^kb$. Thus $k=m$ is even, and so $|N| = 2$.
\medskip

\noindent (3): Let $M = C_n \cap \kr(\pi)$. By Lemma~\ref{ker},
$L(M)^\psi$ is contained in $L(D_n)$. If $|M| > 2,$ then $L(M)^\psi
= L(M)$ follows  by \eqref{eq1}, hence $M \le N$.  This gives $|N| >
2,$ which is a contradiction with (2). Therefore, $|M| \le 2,$ and
so $|\kr(\pi)| \le 4$. We finish the proof of (3) by showing that
$\{1,a^m,ba,ba^{m+1}\} \le \kr(\pi)$.

For  $x \in a \langle a^2 \rangle \, \cup \, b \langle a^2 \rangle,$ let
$\chi(x)$ be the smallest non-negative integer such that $p^{\chi(x)}(x) = x^{-1}$.
Then by \eqref{eq1}, $x \in \kr(\pi)$ if and only if $\chi(\psi(x)) = \chi(x)$.
This shows that $a^m$ and $a^{m-1}b = ba^{m+1}$ are in $\kr(\pi)$ if $m$ is odd, and thus so are
$1$ and $ba$.

Let $m$ be even.  Then $L_{a^m} \psi = \psi L_{a^m},$ hence $a^m \in \kr(\pi)$.
It remain to show that $ba \in \kr(\pi)$. Now
$$
a =\psi(b) = \psi(ba a^{-1}) = \psi(ba) \psi^{\pi(ba)}(a^{-1}) = ab \psi^{\pi(ba)}(a^{-1}).
$$
From this, $\psi^{\pi(ba)}(a^{-1}) = b,$ and so $\pi(ba) =1$.
\medskip

\noindent (4): Note that $\kr(\pi) = \{1,a^m,ba,ba^{m+1}\}$. Since
for any $x \in X$, $\pi(x) \equiv \chi(\psi(x)) - \chi(x) +
1(\text{mod } n)$, one can check $\pi(a^{2i}b) \equiv n-4i-1
(\text{mod } n)$ and $\pi(a^{2i+1}) \equiv 4i+3 (\text{mod } n)$ for
any $i=0,\ldots,m-1$. Furthermore for any $j=0,\ldots, m-1$,
$$
\pi(a^{2j+2}) = \pi(a\cdot a^{2j+1}) \equiv
\pi(a^{2j+1})+\pi(\psi(a^{2j+1}))+\pi(\psi^2(a^{2j+1})) \equiv 4j+5
 (\text{mod } n).$$
Therefore, $\Pi = \{2i+1 : i \in \{0,1,\ldots,m-1\} \}$ and if $4 \mid n,$ then
$$ \pi(x) \equiv -1(\text{mod }4) \iff x \in a \langle a^2 \rangle \, \cup \, b \langle a^2 \rangle.
$$
\QED

The next  lemma can be derived  in the similar way as above, and
hence its proof is omitted.

\begin{lem}\label{map2}
Let  $n=2m,$ $8 \mid n,$ $\M = \cm(D_n,a \langle a^2 \rangle \, \cup \, b \langle a^2 \rangle, p),$
where  $p =  (b,a,a^{m+2}b,a^3,$ $a^4 b,\ldots,a^{m-2}b,a^{n-1})$. Then the following hold:
\begin{enumerate}[(1)]
\item $\M$ is regular, and its associated skew-morphism $\psi$ is given by
$$
\psi(a^j) = \begin{cases} a^{\frac{j}{2}m-j} & \mbox{ if $j$ is even} \\  a^{j+1+\frac{j+1}{2}m}b & \mbox{ if $j$ is odd} \end{cases} \quad
\psi(a^jb) = \begin{cases}
a^{j+1+\frac{j}{2}m} & \mbox{ if $j$ is even} \\
a^{\frac{j+1}{2}m-j}b & \mbox{  if $j$ is odd}.
\end{cases}
$$
\item Let $N$ be the subgroup of $C_n$ such that $L(N)$ is the core of $L(C_n)$ in $\aut(\M)$. Then $|N| = 2$.
\item The associated power function $\pi$ has kernel $\kr(\pi) \cong \Z_2^2$.
\item Let $\Pi = \{ \pi(g) : g \in D_n\}$. Then $\Pi = \{2i+1 : i \in \{0,1,\ldots,m-1\} \},$ and
\begin{equation}\label{eq10}
\pi(x) \equiv -1(\text{mod }4) \iff x \in a \langle a^2 \rangle \, \cup \, b \langle a^2 \rangle.
\end{equation}
\end{enumerate}
\end{lem}

\section{Proof of Theorem~\ref{main}}

In this section we set
\begin{enumerate}[-]
\item $\M = \cm(D_n,X,p)$ is a regular Cayley map;
\item $\psi$ and $\pi$ are the associated skew-morphism and power-function respectively;
\item $N$ is the subgroup of $C_n$ such that $L(N)$ is the core of $L(C_n)$ in
$\aut(\M)$;
\item $n=2m,$  $m \ge 1,$ and $T = \langle a^m \rangle$ is the subgroup of $C_n$ of order
$2$.
\end{enumerate}

\begin{lem}\label{lem1}
Let $n \ge 8,$ $|\kr(\pi)|=4$ and $N$ be non-trivial. Now the
following hold:
\begin{enumerate}[(1)]
\item $N=T$ and $D_n/T$ is a block system for $\aut(\M)$.
\item $|\kr(\pi^{D_n/T})|=4$.
\item $m$ is even.
\end{enumerate}
\end{lem}

\proof We set $G = \aut(\M)$ and $\Gamma = \cay(D_n,X)$.  The cases (1)-(3) are considered separately. \medskip

\noindent (1): Since $N \leq \kr(\pi)$, $|N|=2$ by  Theorem
\ref{min} and thus $T = N$. Since $N \trianglelefteq G,$ $T
\trianglelefteq G$ also holds, and (1) follows. \medskip

\noindent (2):  Let $K$ be the kernel of the action of $\aut(\M)$ on
$D_n/T$. Let $Y_1,Y_2$ be two $T$-cosets such that $(Y_1,Y_2)$ is an
arc of $\Gamma/T$. Let $S$ be the set of arcs from $Y_1$ to $Y_2$ in
$\Gamma$. It is easily seen that $K$ is regular on $S,$ hence $K
\cong \Z_2^i$ for  $i \in \{1,2\}$. Let $H \le C_n$ such that $T \le
H$ and $K L(H)/K$ is the core of $K L(C_n)/K$ in $G/K$. This implies
that $K L(H) \trianglelefteq G$. Define the subgroup $M = \langle
x^2 : x  \in K L(H) \rangle.$ Clearly, this is characteristic in $K
L(H)$. As $K \cong \Z_2^i, i \in \{1,2\}$ and $L(T) \le K,$ the
group $K L(H)$ is $L(H)$ or it can be written as a semidirect
product $L(H) \rtimes \Z_2$. This implies that $M=L(H^+),$ where
$H^+ < H$ and $|H : H^+|=2$.   Since $L(H^+)$ is characteristic in
$K L(H)$ and $K L(H)$ is normal in $G$, $L(H^+)$ is normal in $G$.
Thus $L(H^+) \le L(T),$ $|H| \le 4,$ and so $|\core_{G/K}(K
L(C_n)/K)| \le 2$. Using also that $G^{D_n/T} \cong G/K$ and $K
L(C_n)/K \cong L(C_n)^{D_n/T},$ we conclude that the core of
$L(C_n)^{D_n/T}$  in $G^{C_n/T}$ has order at most $2$.
 Since $\aut(\M/T) = G^{D_n/T}$ and
$L(C_n/T) = L(C_n)^{D_n/T},$ this together with Corollary~\ref{cor}
imply that $|\ker(\pi^{D_n/T})| \le 4$. On the other hand, applying
Theorem~\ref{min} to $\M/T,$ we find that $|\ker(\pi^{D_n/T})| \ge
4$ since $n \ge 8$. Hence $|\ker(\pi^{D_n/T})| = 4$.

\medskip

(3):  Let $K$ be the kernel of the action of $\aut(\M)$ on $D_n/T$.
We have shown in the proof of (2) that $|\core_{G/K}(K L(C_n)/K)|
\le 2$.  If equality holds, then $m=|K L(C_n)/K|$ is even. If
$|\core_{G/K}(K L(C_n)/K)|=1,$ then Theorem~\ref{KMM} gives that $m$
is even because $n \ge 8$. The lemma is proved. \QED

\begin{lem} \label{lem2}
Let $n \ge 8$ and $|\kr(\pi)|=4$. Now the following hold:
\begin{enumerate}[(1)]
\item $X = a \langle a^2 \rangle \, \cup \, b \langle a^2 \rangle,$ and $\psi$ switches
the sets
$a \langle a^2 \rangle$  and $b \langle a^2 \rangle$.
\item Let $\Pi = \{ \pi(g) : g \in D_n\}$.
Then $\Pi = \{2i+1 : i \in \{0,1,\ldots,m-1\} \}$.
\item  For any $g \in D_n$ and for any $a^k \in X$,
$\pi(ga^k) =  \pi(g) + \pi(a^k) -1 (\text{mod } n)$.
\item If $4 \mid n,$ then
$\pi(x) \equiv -1(\text{mod }4)$ if and only if
$x \in a \langle a^2 \rangle \, \cup \, b \langle a^2 \rangle$.
\end{enumerate}
\end{lem}

\proof We prove the lemma by induction on $n$. If $n \le 16,$  all
statements can be checked directly, using the catalog of small
regular maps in~\cite{C09}. Therefore assume that $n > 16,$ and that
the lemma holds for any $n'$ such that $8 \le  n' < n$.  We consider
all cases (1)-(4) step by step.
\medskip

\noindent (1): If $N$ is trivial, then (1) follows from
Theorems~\ref{KMM} and \ref{c-free}.

Let $N$ be non-trivial. Lemma~\ref{lem1} together with the induction
hypothesis imply that $X/T = (a \langle a^2 \rangle \, \cup \, b
\langle a^2 \rangle)/T,$ and $\psi^{D_n/T}$ switches the sets $(a
\langle a^2 \rangle)/T$  and $(b \langle a^2 \rangle)/T$. Notice
that for every $x \in a \langle a^2 \rangle \, \cup \, b \langle a^2
\rangle,$ $|Tx \cap X|$ is the same positive constant, say $c$,
which does not depend on $x$. If  $c=1,$ then $|X| = n/2$. Let us
consider  the action of $\aut(\M)$ on the set of right
$L(D_n)$-cosets by right multiplication. This has degree $|\langle
\psi \rangle|$. For any $d \in D_n,$ $L(D_n) \psi L_d = L(D_n)
\psi^{\pi(d)}$. This shows that the orbit of the coset $L(D_n) \psi$
under $L(D_n)$ is equal to $\{ L(D_n) \psi^{\pi(d)} : d \in D_n\},$
and thus it has size $|\Pi|$. Clearly, $L(D_n),$ as a coset, is
fixed by every permutation in $L(D_n)$. All these imply that $|\Pi|
< |\langle \psi \rangle|$. Since $|\kr(\pi)|=4,$ $|\Pi|=n/2,$ and we
may write
$$
\frac{n}{2} = |\Pi| <  |\langle \psi \rangle| = |X| = \frac{n}{2}.
$$
This is a contradiction, and so $c=2,$ $X = a \langle a^2 \rangle \, \cup \, b \langle a^2 \rangle$.
Since $\psi^{D_n/T}$
switches the sets
$(a \langle a^2 \rangle)/T$  and $(b \langle a^2 \rangle)/T$, $\psi$ also switches the sets
$a \langle a^2 \rangle$  and $b \langle a^2 \rangle$.
\medskip

\noindent (2):  If $N$ is trivial, then (2) follows from
Theorem~\ref{c-free} and Lemma~\ref{map1}(4).

Let $N$ be non-trivial. Lemma~\ref{lem1} together with the induction
hypothesis imply that, for any  $Tg \in D_n/T$, $\pi^{D_n/T}(Tg)$ is
odd. Notice that Lemma~\ref{lem1}(3) gives that $4
\mid n,$ and hence $\psi^{D_n/T}$ is of even order. Now, by
Lemma~\ref{quo}(5), for any $g \in D_n$, $\pi(g)$ is also odd, which
implies that $\Pi = \{2i+1 : i \in \{0,1,\ldots,m-1\} \}$.
\medskip

\noindent (3): Since $\psi$ switches the sets
$a \langle a^2 \rangle$  and $b \langle a^2 \rangle$, for any
$x \in b \langle a^2 \rangle$, we obtain by the formula in \eqref{eq1}
that
$$\pi(x)+\pi(\psi(x)) \equiv 2 (\text{mod }n).$$
Using this property and Lemma~\ref{JS-lem}(3), one can say that
for any $g \in D_n$ and $a^k \in X$,
\begin{eqnarray*} \pi(ga^k) & \equiv & \sum_{i=0}^{\pi(g)-1} \pi(\psi^i(a^k)) \\
&\equiv& \pi(a^k)+\left(\pi(\psi(a^k))+\pi(\psi^2(a^k))\right)+\cdots+\left(\pi(\psi^{\pi(g)-2}(a^k))+\pi(\psi^{\pi(g)-1}(a^k))\right) \\
&\equiv& \pi(a^k) + \pi(g)-1 (\text{mod } n) \end{eqnarray*}
because $\pi(g)$ is odd by case (2).
\medskip

\noindent (4):  If $N$ is trivial, then (4) follows from
Theorem~\ref{c-free} and Lemma~\ref{map1}(4).

Let $N$ be non-trivial. Assume at first that $\M / T$ is a regular
Cayley map on $D_n/T$  such that the core of $L(C_n/T)$ is trivial
in $\aut(\M/T)$. Now $\frac{n}{2}=m=2m'$ with odd $m'$ by
Theorem~\ref{c-free}. Since for any $a^kb \in D_n \setminus C_n$,
there exists a group automorphism $\phi$ such that $\phi(a^kb)=b$,
we can assume that $\psi(a^{-1})=b$. Now $\psi(b)=a$ or
$\psi(b)=a^{m+1}$. Suppose that $\psi(b)=a^{m+1}$.
Since $a^m \in N$ and $N$ is normalized by $\psi,$ $a^m \in
\kr(\pi)$ and $\psi(a^m)=a^m$. Then $\psi(a^{m-1})=a^mb$ and
$\psi(a^mb)=a$. These imply that
$$
\pi(a^{-1})=\pi(a^{m-1})=\pi(b) = \pi(a^mb)=m-1 \ \ \mbox{and} \ \ \pi(a)=\pi(a^{m+1})=m+3.
$$
By (3) we find $\pi(g a) \equiv \pi(g)+m+2(\text{mod
}n)$ for every $g \in D_n$. Since $n\equiv 0(\text{mod }4)$ and
$m\equiv 2(\text{mod }4)$, $\pi(g)
\equiv \pi(ga)(\text{mod }4)$ for every $g \in D_n$. This, however,
contradicts (2).
Therefore, $\psi(b)=a,$ and hence
$$\pi(a^{-1})=\pi(a^{m-1})=\pi(b) = \pi(a^mb)=-1 \ \ \mbox{and} \ \ \pi(a)=\pi(a^{m+1})=3.$$
These imply that $\pi(a^2)= 2\pi(a)-1 = 5$. For any $x \in a \langle a^2 \rangle$, $\pi(a^{2}x)=\pi(x)+\pi(a^2)-1 = \pi(x)+4$.
Hence $\{ 3, 7, \ldots, n-1\} \subset \{\pi(x) : x \in X \}$. Since $X = a \langle a^2 \rangle \, \cup \, b \langle a^2 \rangle $ is a union of
some right cosets  of $\kr(\pi) = \{1,a^m,ba,ba^{m+1}\}$, $|\{\pi(x) : x \in X \}| = \frac{|X|}{4} = \frac{n}{4}$. Hence $\{\pi(x) : x \in X \} = \{ 3, 7, \ldots, n-1\}$.

For the remaining case, let $\M / T$ be a regular Cayley map on
$D_n/T$  such that the core of $L(C_n/T)$ is non-trivial in
$\aut(\M/T)$. Thus $n/2$ is divisible by $4,$ see
Lemma~\ref{lem1}(3). On the other hand,  the order of $\psi^{D_n/T}$
is equal to $|X/T|=n/2,$ and hence $4 \mid |\langle \psi^{D_n/T}
\rangle|$. By the induction hypothesis, $\{\pi^{D_n/T}(Tx) : Tx \in
X/T \} = \{ 3, 7, \ldots, \frac{n}{2}-1\}$. Note that for any $Tx
\in X/T$, $\pi^{D_n/T}(Tx) \equiv -1 (\text{mod } 4)$. By
Lemma~\ref{quo}(5) and the fact that $4 \mid |\langle \psi^{D_n/T}
\rangle|$, for any $x \in X$, $\pi(x)\equiv -1 (\text{mod } 4)$.
Since $X$ is a union of some right cosets of $\kr(\pi)$, $\{\pi(x) :
x \in X \} = \{ 3, 7, \ldots, n-1\}$. \QED

\bigskip

\noindent{\bf Proof of Theorem~\ref{main}.} Let $\M$
be the regular Cayley map given in the theorem. If $n \le 6,$ then
one can check, using the catalog of small regular maps~\cite{C09}, that $\M$ is given
in one of cases (1)-(3) in Theorem~\ref{main}. Let $n > 6$.

If $n$ is odd, then it follows from \cite[Theorem~3.2]{KMM13} that
the core of $L(C_n)$ in $\aut(\M)$ is equal to either $L(C_n),$ or
the subgroup of $L(C_n)$ of index $3$. Since $n > 6,$ the core is
non-trivial of odd order, implying that $|\kr(\pi)| \ne 4$, which is
a contradiction. Thus $n=2m$ and $n \ge 8$. We follow the notations
set at the beginning of this section. We proceed by induction on
$n$. If  $n\le 14,$  then the statement can be checked directly,
using the catalog of small regular maps~\cite{C09}. Therefore,
assume that $n \ge 16$ and that the lemma holds for any $n'$ such
that $8 \le n' < n$.
\medskip

We set $G = \aut(\M),$ $\Gamma = \cay(D_n,X),$ and $\Pi = \{\pi(g) :
g \in D_n\}$. Because of Theorem~\ref{c-free}, we may assume that
$N$ is non-trivial. By Lemma~\ref{lem2}(2), $\Pi = \{ 2i+1 : i \in
\{0,1,\ldots,m-1\}$. Now $D_n/T$ is a block system for $G,$ and
$|\kr(\pi^{D_n/T})| = 4$. Since $n \ge 16,$ $m$ is even by
Lemma~\ref{lem1}(3). The induction hypothesis gives that, up to
equivalence of $\M,$ we may write $p^{D_n/T} = (T b,T a,T
a^2b,\ldots, T a^{m-1})$ or $p^{D_n/T} = (T b,T a,T
a^{2+m/2}b,\ldots, T a^{m-1})$. By Lemma~\ref{lem2}(4), there exists
$x \in X$ such that $\pi(x)=n-1$. This implies that $\pi(a^{k})=n-1$
for some $a^k \in X$. Thus $\pi(T a^k) = m-1,$ see
Lemma~\ref{quo}(5). If $m=2m'$ with odd $m'$, then $k$ equals $m-1$
or $n-1$. If $m$ is a multiple of $4$, then  $k \in
\{\frac{1}{2}m-1,m-1,\frac{3}{2}m-1,n-1\}$. In either cases,  there
exists a group automorphism $\phi$ such that $\phi(a^k)=a^{-1}$,
$\phi(\psi(a^k))=b$ and $\phi(\psi^2(a^k))= a$ since  $k$ and $n$
are relative prime and $\psi^2(a^k)=a^{-k}$. Therefore, we can
assume (up to equivalence of $\M$) that
$$
\psi(a^{-1}) = b \text{ and } \psi(b) = a.
$$

\noindent{\bf Case 1.} $p^{D_n/T} = (T b,T a,T a^2b,T a^3,\ldots,T a^{m-2}b,T a^{m-1})$. \medskip

Clearly, $p(a) \in \{a^2b,a^{m+2}b\}$ and $p^2(a) \in \{a^3,a^{m+3}\}$.
Below we go through all possibilities. \medskip

\noindent {\bf Subcase 1.1.}  $p(a) = a^2b$ and $p(a^2b)=a^3$. Then
$\psi(ab) = \psi(a) \psi^{\pi(a)}(b) = a^2b \psi^3(b) = a^2b a^3 =
a^{-1}b$. Thus $a^{-1}b = \psi(ab) = \psi(ba^{-1}) = \psi(b)
\psi^{\pi(b)}(a^{-1}) = a \psi^{-1}(a^{-1}),$ which implies that
$\psi^{-1}(a^{-1}) = a^{-2}b,$ and so $\psi(a^{-2}b) = a^{-1}$.
 Now we have $a^{-1} = \psi(a^{-2} b) = \psi(b a^2) = \psi(b)
\psi^{\pi(b)}(a^2) = a \psi^{-1}(a^2),$ and $a^3 = \psi(a^2 b) =
\psi(b a^{-2}) = \psi(b) \psi^{\pi(b)}(a^{-2}) = a
\psi^{-1}(a^{-2})$. Consequently, $\psi$ switches $a^2$ and
$a^{n-2}$. By this one can prove that $\psi(a^{2i}) = a^{-2i}$ for
every $i \in \{1,\ldots,m\}$. Since $\Pi$ contains only odd numbers,
we find all the remaining values as:
\begin{eqnarray*}
\psi(a^{2i+1} b) &=& \psi(ab a^{-2i}) = \psi(ab) \psi^{\pi(ab)}(a^{-2i}) = a^{-1}b a^{2i} = a^{n-2i-1} b \\
\psi(a^{2i+1})    &=& \psi(a) \psi^{\pi(a)}(a^{2i}) = a^2 b a^{-2i} = a^{2i+2} b \\
\psi(a^{2i}b)      &=& \psi(b a^{-2i}) = \psi(b)\psi^{\pi(b)}\psi(a^{-2i}) = a a^{2i} = a^{2i+1}
\end{eqnarray*}
Thus $\psi$ is the skew-morphism given in Lemma~\ref{map1}(1), and case
(4) follows. \medskip

\noindent{\bf Subcase 1.2.} $p(a) = a^2b$ and $p^2(a)=a^{m+3}$. Then
$\psi(ab) = \psi(a) \psi^{\pi(a)}(b) = a^2b \psi^3(b) = a^2b a^{m+3}
= a^{m-1}b$. Thus $a^{m-1}b = \psi(ab) = \psi(ba^{-1}) = \psi(b)
\psi^{\pi(b)}(a^{-1}) = a \psi^{-1}(a^{-1}),$ which implies that
$\psi^{-1}(a^{-1}) = a^{m-2}b,$ and so $\psi(a^{-2}b) = a^{m-1}$.
Then we may write $a^{m-1} = \psi(a^{-2} b) = \psi(b a^2) = \psi(b)
\psi^{\pi(b)}(a^2) = a \psi^{-1}(a^2)$ and $a^{m+3} = \psi(a^2 b) =
\psi(b a^{-2}) = \psi(b) \psi^{\pi(b)}(a^{-2}) = a
\psi^{-1}(a^{-2})$. Consequently, $\psi$ switches $a^2$ and
$a^{m-2},$ and also  $a^{-2}$ and $a^{m+2}$. Now $a^2 b = \psi(a) =
\psi(a^{-1} a^2) = \psi(a^{-1}) \psi^{\pi(a^{-1})}(a^2) = b
a^{m-2},$ a contradiction. \medskip

\noindent{\bf Subcase 1.3.} $p(a) = a^{m+2}b$ and $p^2(a)=a^3$. Then
$\psi(ab) = \psi(a) \psi^{\pi(a)}(b) = a^{m+2}b \psi^3(b) = a^{m+2}b
a^3 = a^{m-1}b$. Thus $a^{m-1}b = \psi(ab) = \psi(ba^{-1}) = \psi(b)
\psi^{\pi(b)}(a^{-1}) = a \psi^{-1}(a^{-1}),$ which means
$\psi^{-1}(a^{-1}) = a^{m-2}b,$ and so $\psi(a^{m-2}b) = a^{-1}$.
Now $\psi(a^2) = \psi(a^{-1}
a^3)=\psi(a^{-1})\psi^{\pi(a^{-1})}(a^3) = b a^{m+2}b = a^{m-2}$,
which implies $\psi(a^{m+2}) = a^{-2}$. Also, we have $\psi(a^{-2})
= \psi(a^{-1} a^{-1})=\psi(a^{-1})\psi^{\pi(a^{-1})}(a^{-1}) = b
a^{m-2}b = a^{m+2}$, and hence $\psi(a^{m-2}) = a^2$. Furthermore,
$\psi(a^4) = \psi(a^2)\psi^{\pi(a^2)}(a^2) = a^{m-2} a^{m-2} =
a^{-4};$ and similarly, $\psi(a^{-4}) = a^4$. By these one can prove
that $\psi(a^{4i+2}) = a^{-4i-2+m}$ and $\psi(a^{4i}) = a^{-4i}$ for
every $i \in \{1,\ldots,m/2-1\}$. Now for $\psi$ to be well defined,
$m$ should be a multiple of $4$, and hence $8 \mid n$. Using these
values of $\psi$ and the fact that $\Pi$ contains only odd numbers,
we find all the remaining values as:
\begin{eqnarray*}
\psi(a^{2i+1} b) &=& \psi(ab a^{-2i}) = \psi(ab) \psi^{\pi(ab)}(a^{-2i}) = a^{m-1}b a^{im+2i} = a^{(i+1)m-2i-1} b \\
\psi(a^{2i+1})    &=& \psi(a) \psi^{\pi(a)}(a^{2i}) = a^{m+2} b a^{im-2i} = a^{(i+1)m+2i+2} b \\
\psi(a^{2i}b)      &=& \psi(b a^{-2i}) = \psi(b)\psi^{\pi(b)}\psi(a^{-2i}) = a a^{im+2i} = a^{im+2i+1}
\end{eqnarray*}
Thus $\psi$ is the skew-morphism given in Lemma~\ref{map2}(1), and case (5)
follows.
\medskip

\noindent{\bf Subcase 1.4.} $p(a) = a^{m+2}b$ and $p^2(a)=a^{m+3}$.
Then $\psi(ab) = \psi(a) \psi^{\pi(a)}(b) = a^{m+2}b \psi^3(b) =
a^{m+2}b a^{m+3} = a^{-1}b$. Thus $a^{-1}b = \psi(ab) =
\psi(ba^{-1}) = \psi(b) \psi^{\pi(b)}(a^{-1}) = a
\psi^{-1}(a^{-1}),$ which means $\psi^{-1}(a^{-1}) = a^{-2}b,$ and
so $\psi(a^{-2}b) = a^{-1}$. Now we have $a^{-1} = \psi(a^{-2} b) =
\psi(b a^2) = \psi(b) \psi^{\pi(b)}(a^2) = a \psi^{-1}(a^2)$ and
$a^{m+3} = \psi(a^{m+2} b) = \psi(b a^{m-2}) = \psi(b)
\psi^{\pi(b)}(a^{m-2}) = a \psi^{-1}(a^{m-2})$. Consequently, $\psi$
switches $a^2$ and $a^{-2},$ and also  $a^{m+2}$ and $a^{m-2}$. Now,
$a^{m+2} b = \psi(a) =  \psi(a^{-1} a^2) = \psi(a^{-1})
\psi^{\pi(a^{-1})}(a^2) = b a^{-2},$ a contradiction. \medskip

\noindent{\bf Case 2.}  $p^{D_n/T} = (T b,T a,T a^{2+m/2}b,\ldots, T a^{m-1})$. \medskip

Now $\psi^{-1}(a^{-1}) \in \{a^{\frac{m}{2}-2}b,a^{\frac{3m}{2}-2}b \}$ and $\pi(\psi^{-1}(a^{-1}))= 3$.
By Lemma~\ref{lem2}(3),
$$\pi(\psi^{-1}(a^{-1}) a^{-1}) \equiv \pi(\psi^{-1}(a^{-1}))+ \pi(a^{-1})-1 \equiv 1 (\text{mod } n),$$
which implies that $\psi^{-1}(a^{-1}) a^{-1} \in \kr(\pi)$. Notice that $\psi^{-1}(a^{-1}) a^{-1} = a^{\frac{m}{2}-1}b$ or $a^{\frac{3m}{2}-1}b$.
Since $\pi(a^{-1}b)=\pi(ba) \equiv \pi(b)+\pi(a)-1 \equiv 1 (\text{mod } n)$, $a^{-1}b$ also belongs to $\kr(\pi)$.  These imply that $a^{\frac{m}{2}} \in \kr(\pi)$, and hence $|\kr(\pi)| \ge 8$, a contradiction.  \QED


\end{document}